\documentclass[10pt,letterpaper,conference]{ieeeconf}

\usepackage[hyperfootnotes=false]{hyperref}
\hypersetup{ 		
    pdfauthor={Tyler Summers},        
    colorlinks=true,        
    linkcolor=NavyBlue,
    urlcolor=NavyBlue,
    citecolor=Mahogany     		
}

\usepackage[usenames,dvipsnames]{color}
\usepackage{epstopdf}
\usepackage{tabularx}
\usepackage{colortbl}
\usepackage{wrapfig}
\usepackage{cite}

\usepackage{amsmath,amssymb,mathrsfs,dsfont}
\usepackage{algorithm,algorithmic}	
\usepackage{todonotes}

\newtheorem{theorem}{Theorem}
\newtheorem{corollary}{Corollary}
\newtheorem{lemma}{Lemma}

\IEEEoverridecommandlockouts

\title{Performance bounds for optimal feedback control in networks}
\author{Tyler Summers and Justin Ruths\thanks{The authors are with the Department of Mechanical Engineering, University of Texas at Dallas, Richardson, TX, 75080 USA. E-mail: \{tyler.summers,jruths\}@utdallas.edu, The work of T. Summers was sponsored by the Army Research Office and was accomplished under Grant Number: W911NF-17-1-0058. The views and conclusions contained in this document are those of the authors and should not be interpreted as representing the official policies, either expressed or implied, of ARO or the U.S. Government.} }

\begin{document}
\maketitle

\begin{abstract}
Many important complex networks, including critical infrastructure and emerging industrial automation systems, are becoming increasingly intricate webs of interacting feedback control loops. A fundamental concern is to quantify the control properties and performance limitations of the network as a function of its dynamical structure and control architecture. We study performance bounds for networks in terms of optimal feedback control costs. We provide a set of complementary bounds as a function of the system dynamics and actuator structure. For unstable network dynamics, we characterize a tradeoff between feedback control performance and the number of control inputs, in particular showing that optimal cost can increase exponentially with the size of the network. We also derive a bound on the performance of the worst-case actuator subset for stable networks, providing insight into dynamics properties that affect the potential efficacy of actuator selection.
We illustrate our results with numerical experiments that analyze performance in regular and random networks. 
\end{abstract}

\section{Introduction}
Recent spectacular advances in computation and communication technologies are transforming our ability to control complex networked systems. Critical infrastructure, industrial automation systems, and many other technological and social networks crucial to modern society are becoming increasingly intricate webs of interacting feedback loops. As this complexity increases, a fundamental concern is to quantify the control properties and performance limitations of the network as a function of its dynamical structure and control architecture.

A variety of metrics can be used to quantify notions of network controllability. Significant recent research has been devoted to studying connections between such notions and the structural properties of the network, and to studying algorithms for designing networks with good controllability properties. One broad line of work has focused on classical binary controllability metrics based on Kalman rank \cite{liu2011controllability,rajapakse2011dynamics,cowan2012nodal,nepusz2012controlling,wang2012optimizing,ruths2014control,olshevsky2014minimal,pequito2016}. Another line of work has focused on metrics based on the Gramian \cite{pasqualetti2014controllability,summers2014submodularity,summers2014optimal,yannature2015,tzoumas2016}. These binary and open-loop notions fail to capture essential feedback and robustness properties, and other recent work has considered more general optimal control and estimation metrics \cite{summers2016actuator,polyak2013lmi,munz2014sensor,dhingra2014admm,zhang2017sensor}.

An important part of understanding network controllability in terms of any metric is expressing fundamental performance limitations. A clear understanding of performance limitations can set practical expectations and guide the design and analysis of network control architectures. Recent work on performance limitations and network controllability include \cite{pasqualetti2014controllability} in the context of the Gramian, and \cite{zhang2017sensor} in the context of sensor selection and Kalman filtering. However, to our best knowledge no such studies have been done in a network context for more general optimal control metrics.


The main contributions of the paper are as follows. First, we derive a set of complementary performance bounds for dynamical networks in terms of optimal feedback control performance. Our bounds are based on the system dynamics and characterize a tradeoff between achievable feedback control performance and the actuator structure. In particular, we show that when the network dynamics are unstable, the optimal cost can increase exponentially with the size of the network for any fixed-size actuator set. The implication of this is that feedback control costs may be extremely high even with an optimal selection of a fixed number of actuators. Second, we derive bounds for the worst-case performance of actuator selection algorithms as a function of the system dynamics for stable systems, showing that greedy algorithms for actuator selection cannot produce arbitrarily bad selections. Finally, we illustrate our results by analyzing performance in regular and random networks. Even though the best case actuator selection may result in large feedback control cost for unstable networks and the worst case actuator selection cannot be arbitrarily bad for stable networks, we show that actuator selections can have a significant effect. 

The rest of the paper is organized as follows. Section II provides preliminaries on optimal control. Section III develops performance bounds based on the system dynamics for stable and unstable networks. Section IV presents illustrative numerical experiments. Section V concludes.

\textbf{Notation.} The eigenvalues of a square matrix $A$ are denoted by $\lambda_i(A)$ and ordered $|\lambda_{max}(A)| = |\lambda_1(A)| \geq |\lambda_2(A)| \geq \cdots \geq |\lambda_n(A)| = |\lambda_{min}(A)|$. The singular values of a matrix $F$ are denoted by $\sigma_i(F)$ and ordered $\sigma_1(F) \geq \sigma_2(F) \geq \cdots \geq \sigma_n(F)$. The condition number of a matrix $V$ is denoted $\text{cond}(V)$.

\section{Preliminaries}
We begin by formulating actuator selection problems based on optimal feedback control performance for linear dynamical systems with quadratic cost functions.  The development follows \cite{summers2016actuator}.

\subsection{Deterministic optimal feedback control} \label{lqr_prob}
The network dynamics are modeled by the discrete-time linear dynamical system evolving on a graph $\mathcal{G} = (\mathcal{V},\mathcal{E})$
\begin{equation} \label{dtdlds}
x_{t+1} = Ax_t + B_S u_t, \quad t = 0,...,T,
\end{equation}
where $x_t \in \mathbf{R}^n$ is the system state at time $t$, $u_t \in \mathbf{R}^{ |S| }$ is the input at time $t$, and $A$ is the network dynamics matrix, which encodes the weighted connections in the underlying graph $\mathcal{G}$ and we assume to be invertible throughout. Let $\mathcal{B} = \{ b_1,..., b_M \}$ be a finite set of $n$-dimensional column vectors associated with possible locations for actuators that could be placed in the network to affect the dynamics of nodes in the graph. For any subset $S \subset \mathcal{B}$, the input matrix $B_S$ comprises the columns indexed by $S$, i.e., $B_S = [b_{s_1}, ..., b_{s_{|S|}}] \in \mathbf{R}^{n \times |S|}. $

We first consider an optimal open-loop linear quadratic regulator performance index associated with an input sequence $\mathbf{u} = [u_0^T,...,u_{T-1}^T]^T$. The optimal cost function is 
$$ V_{LQR}^*(S,x_0) = \min_\mathbf{u}  \sum_{t=0}^{T-1} (x_t^T Q x_t + u_t^T R_S u_t) + x_T^T Q_T x_T, $$ 
where $Q \succeq 0$ and $Q_T \succ 0$ are state and terminal cost matrices and $R_S \succ 0$ is an input cost matrix associated with actuator subset $S$.
This standard least squares problem has the solution
\begin{equation}
\begin{aligned} \label{openloopcost}
V_{LQR}^*(S, x_0) &= x_0^T  \underbrace{G^T (I +  H  \mathbf{B}_S \mathbf{B}_S^T  H^T)^{-1} G}_{P_0} x_0
\end{aligned}
\end{equation}
where $$H = \text{diag}(Q^{\frac{1}{2}})  \left[\begin{array}{ccccc}0 & 0 & 0 & \cdots & 0 \\I & 0 & 0 & \cdots & 0 \\A & I & 0 & \cdots & 0 \\A^2 & A & I & \ddots & 0 \\\vdots & \vdots & \vdots & \ddots & 0 \\A^{T-1} & A^{T-2} & \cdots & A & I\end{array}\right], \quad $$ $$ G = \text{diag}(Q^{\frac{1}{2}}) \left[\begin{array}{c}I \\A \\A^2 \\\vdots \\A^T\end{array}\right], \quad \mathbf{B}_S = \text{diag} (B_S R_S^{-\frac{1}{2}}). $$  
Alternatively, dynamic programming can be used to compute the optimal cost matrix $P_0$ via the backward Riccati recursion
\begin{equation} \label{riccatirecursion}
P_{t-1} = Q + A^T P_t A - A^T P_t B_S (R_S + B_S^T P_t B_S)^{-1} B_S^T P_t A,
\end{equation}
for $t = T,...,1$ with $P_T = Q_T$. The infinite horizon cost matrix $P$ can be computed from the limit of the recursion, resulting in the algebraic Riccati equation
\begin{equation} \label{riccatiequation}
P = Q + A^T P A - A^T P B_S (R_S + B_S^T P B_S)^{-1} B_S^T P A.
\end{equation}

The optimal cost function \eqref{openloopcost} quantifies feedback control performance as a function of the actuator subset and the initial state. Our performance bounds will be expressed in terms of worst-case and average values of this cost over initial states. In particular, we define
\begin{equation}
\begin{aligned}
\hat J_{LQR}(S) &= \max_{\Vert x_0 \Vert = 1} V^*(S, x_0) = \lambda_{max}(P_0) \\
 J_{LQR}^*(S) &= \mathbf{E}_{x_0} V^*(S, x_0) = \mathbf{tr} [P_0 X_0],
\end{aligned}
\end{equation}
where $\hat J_{LQR}(S)$ represents a worst-case cost and $ J_{LQR}^*(S)$ represents an average cost over a distribution of initial states with zero-mean and finite covariance $X_0$.

\textbf{Actuator selection.} The mappings $J_{LQR}^* : 2^\mathcal{B} \rightarrow \mathbf{R}$ and $\hat J_{LQR} : 2^\mathcal{B} \rightarrow \mathbf{R}$ shown above are set functions that map actuator subsets to optimal feedback control performance. We pose set function optimization problems to select a $k$-element subset of actuators to optimize control performance
\begin{equation}
\min_{S \subset V, \ |S| = k} \hat J_{LQR}(S), \quad \quad \min_{S \subset V, \ |S| = k} J_{LQR}^*(S). 
\end{equation}
Our performance bounds will also be expressed and interpreted in terms of actuator subset selections.

\subsection{Stochastic optimal control} \label{lqg_prob}
A more general for network dynamics is the stochastic linear system
\begin{equation} \label{dtstochlds}
x_{t+1} = Ax_t + B_S u_t + w_t, \quad t = 0,...,T,
\end{equation}
where $\{w_t\}$ is an identically and independently distributed Gaussian process with $w_t \sim \mathcal{N}(0, W)$ that models random disturbances that affect the network dynamics. 
For stochastic systems, we optimize expected cost over the set $\Pi$ of causal, measurable state feedback \emph{policies} $\pi: \mathbf{R}^n \rightarrow \mathbf{R}^m$
$$ V_{LQG}^*(S,x_0) = \min_{\pi \in \Pi} \mathbf{E}_w  \sum_{t=0}^{T-1} (x_t^T Q x_t + u_t^T R_S u_t) + x_T^T Q_T x_T. $$ 
Via dynamic programming, the optimal cost function is quadratic and given by
\begin{equation} \label{finitelqg}
V_{LQG}^*(S,x_0) = x_0^T P_0 x_0 + \sum_{t=1}^T \mathbf{tr} P_t W,
\end{equation}
where the $P_t$ for $t=0,...,T$ are generated by \eqref{riccatirecursion}, the same recursion as in the deterministic problem. 
In an infinite horizon setting, the steady state average stage cost is given by
\begin{equation} \label{inflqg}
J^\infty_{LQG}(S) = \mathbf{tr}(PW),
\end{equation}
where $P$ is the positive semidefinite solution of \eqref{riccatiequation}. As before, we can also define worst-case and average values for \eqref{finitelqg} and \eqref{inflqg} and pose corresponding set function optimization problems for actuator selection. Here, the optimal cost depends on both the initial state distribution and the disturbance distribution.

\section{Bounds on optimal feedback control performance}
We now develop a set of complementary bounds on the optimal feedback control performance in networks as a function of the system dynamics and the actuator subset $S$. We start with a worst-case lower bound for the best possible actuator subset selection for unstable networks. This result shows that the optimal cost can be exponentially large even with the best fixed-size set of actuators. We then derive a worst-case upper bound for the worst possible actuator subset selection for stable networks. This result shows that even the worst set of actuators cannot have arbitrarily bad performance. Our results are inspired by bounds for the controllability Gramian \cite{pasqualetti2014controllability} and an analogous bound for the Kalman filter in the context of sensor selection for state estimation \cite{zhang2017sensor}. 

\subsection{Performance bound for unstable network dynamics}
We begin with the following performance bound on optimal feedback control of networks with unstable open-loop network dynamics. To simplify the exposition, we will assume throughout this subsection that $\mathcal{B} = \{e_1,...,e_n \}$, the canonical set of unit vectors (i.e., each input signal affects the dynamics of a single node), and that $R_S = I$, $\forall S$. However, it is straightforward to generalize the results to arbitrary input vectors and cost matrices. We focus here on the infinite horizon cost given by the algebraic Riccati equation \eqref{riccatiequation}.
\begin{theorem} \label{thm:unstablebound}
Consider a network $\mathcal{G} = (\mathcal{V},\mathcal{E})$ with dynamics matrix $A$ and input set $S \subset \mathcal{B}$. Suppose that $A$ is Schur unstable and let $\lambda_{max}(A) > 1$ denote the eigenvalue of $A$ with maximum magnitude. Suppose further that $A$ is diagonalizable by the eigenvector matrix $V$, and for any $\eta \in (1, \lambda_{max}(A)]$ define
$$ \bar n =  | \{ \lambda : \lambda \in \text{spec}(A), |\lambda| \geq \eta \} |. $$
For all $\eta \in (1, \lambda_{max}(A)]$ and for any $Q \succeq 0$ such that $(A, Q^{\frac{1}{2}})$ is detectable, it holds 
\begin{equation} \label{unstablebound}
\lambda_{max}(P) \geq \text{cond}^{-2}(V) \frac{\eta^2 - 1}{\eta^2} \eta^{2(\frac{\bar n}{|S|} - 1)},
\end{equation}
where $P$ is the optimal closed-loop cost matrix that satisfies the algebraic Riccati equation \eqref{riccatiequation}.
\end{theorem}
\begin{proof}
We first make a connection between the optimal cost matrix for small $Q$ and a controllability Gramian associated with the \emph{inverse} of the dynamics matrix. 
Applying the Woodbury matrix identity to the Riccati recursion \eqref{riccatirecursion} yields
\begin{equation}
P_{t-1} = Q + A^T (P_t^{-1} + B_S R_S^{-1} B_S^T)^{-1} A.
\end{equation}
As $Q \rightarrow 0$ the inverse cost matrix satisfies
\begin{equation}
P_{t-1}^{-1} = A^{-1} (P_t^{-1} + B_S R_S^{-1} B_S^T) A^{-T}.
\end{equation}
Defining $X_{T-t} = P_t^{-1} + B_S B_S^T$, setting $R_S=I$, and rearranging, we obtain the recursion
\begin{equation} \label{grammian1}
X_{\tau + 1} = A^{-1} X_\tau A^{-T} + B_S B_S^T, \quad \tau=0,...,T-1
\end{equation}
with $X_0 = P_T^{-1} + B_S B_S^T = Q_T^{-1} + B_S B_S^T$.  This gives 
\begin{equation} \label{grammian2}
X_T = \underbrace{\sum_{\tau = 0}^{T-1} (A^{-1})^\tau B_S B_S^T  (A^{-T})^\tau}_{\bar X_T}  + (A^{-1})^T Q_T^{-1} (A^{-T})^T.
\end{equation}
We see that $\bar X_T$ is the $T$-stage controllability Gramian associated with the system $(A^{-1}, B_S)$. Then directly applying Theorem 3.1 of \cite{pasqualetti2014controllability}, for any $\mu \in [\lambda_{min}(A^{-1}), 1)$ and any $T \in [1,\infty)$  it holds that 
\begin{equation} \label{grammianbound}
\lambda_{min}(\bar X_T) \leq \text{cond}^2(V) \frac{ \mu^{2(\frac{\bar n}{|S|} - 1)}}{1-\mu^2} 
\end{equation}
where 
$$\bar n =  | \{ \lambda : \lambda \in \text{spec}(A^{-1}), |\lambda| \leq \mu \} |.$$ Defining $ \eta = 1/\mu$, we see that 
$$\bar n =  | \{ \lambda : \lambda \in \text{spec}(A), |\lambda| \geq \eta \} |$$
and $\eta \in (1, \lambda_{max}(A)]$. Since $P_0^{-1} = X_T - B_S B_S^T$, it follows that $\lambda_{min}(P_0^{-1}) \leq \lambda_{min}(X_T)$. Since $A$ has at least one unstable eigenvalue, then $A^{-1}$ has at least one stable eigenvalue, and in this direction the minimum eigenvalue of the second term in \eqref{grammian2} approaches zero as $T \rightarrow \infty$ for any fixed $Q_T \succ 0$, so that $\lim_{T\rightarrow \infty} \lambda_{min}(X_T) = \lambda_{min}(\bar X_T)$. Thus from \eqref{grammianbound} we have in the limit as $T \rightarrow \infty$
\begin{equation}
\begin{aligned}
\lambda_{max}(P) &\geq 1/\lambda_{min}( \bar X_T) \\
& \geq \text{cond}^{-2}(V)(1-\mu^2)  \mu^{-2(\frac{\bar n}{|S|} - 1)}
\end{aligned}
\end{equation}
Substituting $\mu = 1/\eta$ yields the expression \eqref{unstablebound}.

Finally, this analysis for small $Q$ accounts only for input energy costs and not for state regulation costs. It is clear from the structure of the recursion \eqref{riccatirecursion} (and from standard comparison lemmas; see, e.g., Chapter 13 in \cite{lancaster1995algebraic}) that for any $Q \succeq 0$ such that $(A, Q^{\frac{1}{2}})$ is detectable the costs can only increase. In particular, if $P^{Q\rightarrow 0}$ denotes the solution to \eqref{riccatiequation} for small $Q$ and $P^{Q}$ the solution for any $Q \succeq 0$ such that $(A, Q^{\frac{1}{2}})$ is detectable, then $P^{Q} \succeq P^{Q\rightarrow 0}$. Thus, the bound remains valid for any such choice of $Q$.
\end{proof}

\textbf{Discussion.} Although our result is inspired by and utilizes a bound on the minimum eigenvalue of the controllability Gramian in \cite{pasqualetti2014controllability}, we emphasize that it is not a trivial inversion of their bound. The Gramian quantifies input energy required for state transfer from the origin, so that a limiting feature of the dynamics is stable modes. In contrast, the optimal cost matrix quantifies input energy and state regulation costs (to the origin) for feedback control, and a limiting feature of the dynamics is unstable modes. Of course, this is as expected, but one arrives at significantly different conclusions about how easy or difficult it is to control a network, depending on which quantitative notion of network controllability is used. Our bound involves a fundamental \emph{closed-loop, feedback} notion of controllability. 

The bound expresses a fundamental performance limitation for feedback control of networks with unstable dynamics. Specifically, if the number of unstable modes grows, then the feedback control costs increase exponentially for any fixed-size set of actuators, even if they are optimally placed in the network. An immediate corollary (cf. Corollary 3.3 in \cite{pasqualetti2014controllability}) is that in order to guarantee a bound on the optimal control cost, the number of actuators must be a linear function of the number of unstable modes, even though a single actuator may suffice to stabilize the network dynamics in theory. As in \cite{pasqualetti2014controllability} and as we will see in our numerical experiments, the bound is loose in many cases, so that very large costs can be incurred even with a small number of unstable modes.

There are several ways the bound might be improved. It only accounts for the number of actuators, and not how effectively they control crucial state space dynamics. It could be improved, for example, by incorporating the angles that the input vectors make with the left eigenvectors of the dynamics matrix. Furthermore, the bound excludes the contribution of state regulation costs, so a sharper bound could be developed that includes and distinguishes both. It would also be interesting to explore possible connections with classical frequency domain performance limitations, such as Bode sensitivity theorems.


We conclude this subsection with a corollary that expresses a simplified bound for symmetric networks.
\begin{corollary}
Consider a network $\mathcal{G} = (\mathcal{V},\mathcal{E})$ with dynamics matrix $A$ and input set $S \subset \mathcal{B}$. Suppose that $A$ is Schur unstable and symmetric. Let $\lambda_{max}(A) > 1$ denote the eigenvalue of $A$ with maximum magnitude and $\bar \lambda_{u}(A) > 1$ denote the unstable eigenvalue of $A$ with minimum magnitude. For any $Q \succeq 0$ such that $(A, Q^{\frac{1}{2}})$ is detectable, it holds 
\begin{equation}
\begin{aligned}
\lambda_{max}(P) \geq \max & \left\{\frac{\lambda_{max}(A)^2 - 1}{\lambda_{max}(A)^2}, \right. \\ 
& \quad \left. \frac{\bar \lambda_{u}(A)^2 - 1}{\bar \lambda_{u}(A)^2} \bar \lambda_{u}(A)^{2(\frac{\bar n}{|S|} - 1)}  \right\} .
\end{aligned}
\end{equation}
\end{corollary}

\begin{proof}
To obtain the bound for the first term, consider the controllability Gramian $\bar X_T$ relating to the inverse cost matrix for small $Q$ in \eqref{grammian2}. Let $\bar X_{T,\mathcal{B}}$ be the Gramian for $S=\mathcal{B}$. Since $\bar X_T \preceq \bar X_{T,\mathcal{B}}$, it follows that $\lambda_{min}(\bar X_T) \leq \lambda_{min}(\bar X_{T,\mathcal{B}})$. We then have
\begin{equation}
\begin{aligned}
\lambda_{min}(\bar X_{T,\mathcal{B}}) = \lambda_{min} \left( \sum_{\tau=0}^{T-1} A^{-2\tau} \right) = \frac{1 - \lambda_{min}(A^{-1})^{2T}}{1 -  \lambda_{min}(A^{-1})^{2}} \\
\Rightarrow \lim_{T\rightarrow \infty} \lambda_{min}(\bar X_{T,\mathcal{B}}) = \frac{\lambda_{max}(A)^2}{\lambda_{max}(A)^2 - 1}.
\end{aligned}
\end{equation}
The first part then follows since as $T \rightarrow \infty$ we have $\lambda_{max}(P) \geq 1/\lambda_{min}(\bar X_T) \geq 1/\lambda_{min}(\bar X_{T,\mathcal{B}})$. The bound for the second term follows from Theorem \ref{thm:unstablebound} with $\eta = \bar  \lambda_{u}(A) > 1$ and since the symmetric dynamics matrix admits an orthonormal eigenvector matrix $V$ with $\text{cond}(V) =1$.
\end{proof}

\subsection{Performance bound for stable network dynamics}
Next we derive a complementary performance bound for stable network dynamics. It establishes a worst case performance bound for actuator subsets produced by any selection algorithm and quantifies how the difference between the best and worst possible actuator subsets depends on the network dynamics. This analysis is inspired by analogous results for sensor selection in the context of a state estimation metric involving the Kalman filtering error covariance matrix \cite{zhang2017sensor}. We focus here on the infinite horizon cost given by the solution to the algebraic Riccati equation \eqref{riccatiequation}, though it is also straightforward to derive for finite horizon costs.

We consider the following ratio
\begin{equation}
r(P) = \frac{\mathbf{tr}(P_{worst})}{\mathbf{tr}(P_{opt})},
\end{equation}
where $P_{worst}$ and $P_{opt}$ are the solutions to the algebraic Riccati equation \eqref{riccatiequation} corresponding to the optimal and worst $k$-element selection of actuators.

Analogous to the sensor information matrix defined in \cite{zhang2017sensor}, we also define the following \emph{actuator influence matrix} corresponding to an actuator subset $S \subseteq \mathcal{B}$
\begin{equation}
R(S) := B_S R_S^{-1} B_S^T.
\end{equation}

To prove the result, we will utilize the following lemmas.
\begin{lemma}[\cite{komaroff1994iterative}] \label{riccineq}
The solution $P \succeq 0$ of \eqref{riccatiequation} with $Q \succ 0$ satisfies $P \succeq A^T(Q^{-1} + R(S))^{-1} A + Q$.
\end{lemma}

\begin{lemma}[\cite{horn1985matrix}] \label{traceeigineq}
For symmetric matrices $Y,Z \in \mathbf{R}^{n \times n}$, there holds $\lambda_n(Y+Z) \geq \lambda_n(Y) + \lambda_n(Z)$, $\lambda_1(Y+Z) \leq \lambda_1(Y) + \lambda_1(Z)$, and $\lambda_n(Y) \mathbf{tr}(Z) \leq \mathbf{tr}(YZ)  \leq \lambda_1(Y) \mathbf{tr}(Z)$. 
\end{lemma}

\begin{lemma}[\cite{liu2011open}] \label{atrans}
A square matrix $A \in \mathbf{R}^{n \times n}$ is Schur stable if and only if there exists a nonsingular matrix $T$ such that $\sigma_1(TAT^{-1}) < 1$. 
\end{lemma}

Based on the similarity transformation $T$ in Lemma \ref{atrans}, we define a positive constant which will appear in our bound:
\begin{equation} \label{alphaa}
\alpha_A = \frac{\sigma_1^2(T)}{\sigma_n^2(T)(1 - \sigma^2_1(TAT^{-1}))}.
\end{equation}

\begin{theorem} \label{thm:stablebound}
Let $\mathcal{R} = \{ R(S) \mid S \subset \mathcal{B}, \ |S| \leq k \}$ be the set of all actuator influence matrices for actuator subsets with $k$ or fewer elements. Let $\lambda_1^{max} := \max\{ \lambda_1(R) \mid R \in \mathcal{R} \}$. Suppose the dynamics matrix $A$ is stable and $Q \succ 0$. Then the cost ratio satisfies
\begin{equation} \label{stablebound}
r(P) \leq \frac{\alpha_A (1 + \lambda_1^{\max} \lambda_n(Q)) \mathbf{tr}(Q)}{  \sigma_n^2(A) \lambda_n(Q) + (1 + \lambda_1^{\max} \lambda_n(Q))  \mathbf{tr}(Q) }
\end{equation}
\end{theorem}
\begin{proof}
Our proof follows along the lines of the analogous proof of Theorem 3 in \cite{zhang2017sensor}.
We begin by deriving an upper bound for $\mathbf{tr}(P_{worst})$, based on the fact that for stable systems the cost is finite even without any actuation. Specifically, with no actuators ($S = \emptyset$) the algebraic Riccati equation \eqref{riccatiequation} reduces to the Lyapunov equation
$$ P^\emptyset = A^T P^\emptyset A + Q. $$
Since $A$ is stable, from Lemma \ref{atrans} there exists a nonsingular similarity transformation $T$ satisfying $\sigma_1(TAT^{-1}) < 1$. Defining $\bar P = T P^\emptyset T^T$, $\bar Q = T Q T^T$, and $D = TAT^{-1}$, we have $\bar P = D \bar P D^T + \bar Q$. Using trace and eigenvalue interlacing properties for sums of symmetric matrices from Lemma \ref{traceeigineq}, there holds $ \mathbf{tr}(D \bar P D^T) =  \mathbf{tr}(D^T D \bar P) \leq \sigma_1^2(D) \mathbf{tr}(\bar P)$ so that 
$ \mathbf{tr}(\bar P) \leq \frac{\mathbf{tr}(\bar Q) }{1 - \sigma_1^2(D)}.$
Similarly, we have
$\mathbf{tr}(\bar P) = \mathbf{tr}(T^T T P^\emptyset) \geq \sigma^2_n(T) \mathbf{tr}(P^\emptyset)$ and $\mathbf{tr}(\bar Q) = \mathbf{tr}(T^T T Q) \leq \sigma_1^2(T) \mathbf{tr}(Q)$. Putting it all together yields
\begin{equation}
\begin{aligned}
\mathbf{tr}(P_{worst}) &\leq \mathbf{tr}(P^\emptyset) \\
& \leq \frac{\sigma_1^2(T)}{\sigma_n^2(T)} \frac{\mathbf{tr}(Q)}{1-\sigma^2_1(D)} = \alpha_A \mathbf{tr}(Q).
\end{aligned}
\end{equation}
where $\alpha_A$ is the constant defined in \eqref{alphaa}.

We now provide a lower bound for $P_{opt}$. For any $k$-element actuator subset $S$, there holds
\begin{equation}
\begin{aligned}
\mathbf{tr}(P) &\geq \mathbf{tr}(A^T (Q^{-1} + R(S) )^{-1} A + Q) \\
& \geq \lambda_n(A A^T) \mathbf{tr}(Q^{-1} + R(S) )^{-1} + \mathbf{tr}(Q) \\
& = \sigma_n^2(A) \sum_{i=1}^n \frac{1}{\lambda_i(Q^{-1} + R(S)) } + \mathbf{tr}(Q) \\
& \geq \frac{n \sigma_n^2(A) }{\lambda_1(Q^{-1} + R(S))} + \mathbf{tr}(Q) \\ 
& \geq \frac{n \sigma_n^2(A) }{\lambda_1(Q^{-1}) + \lambda_1(R(S))} + \mathbf{tr}(Q) \\
& \geq \frac{n \sigma_n^2(A) }{\frac{1}{\lambda_1(Q)} + \lambda_1^{max}} + \mathbf{tr}(Q).
\end{aligned}
\end{equation}
The first inequality follows from Lemma \ref{riccineq}, and the second and fourth from Lemma \ref{traceeigineq}. Since the bound above holds for any $k$-element actuator subset, it also holds for the optimal $k$-element selection.  

Finally, the bound \eqref{stablebound} is obtained by combining the upper bound for $\mathbf{tr}(P_{worst})$ and the lower bound for $\mathbf{tr}(P_{opt})$.
\end{proof}

We also state the following corollary, which provides a simplified bound for stable and normal dynamics matrices.
\begin{corollary} \label{stableboundcor}
If the system dynamics matrix $A$ is Schur stable, then $r(P) \leq \alpha_A$, where $\alpha_A$ is the constant defined in \eqref{alphaa} that depends only on the network dynamics matrix. Moreover, if $A$ is also normal, i.e., $A^TA = AA^T$, then 
\begin{equation} \label{stablesimplebound}
 r(P) \leq \frac{1}{1-\lambda_1^2(A)}.
 \end{equation}
\end{corollary}
\begin{proof}
Since the denominator in the bound \eqref{stablebound} is lower bounded by $ (1 + \lambda_1^{\max} \lambda_n(Q))  \text{trace}(Q)$, a looser bound $r(P) \leq \alpha_A$ is obtained that only depends on the system dynamics, and not on the cost matrix $Q$. In addition, if $A$ is normal, its singular values are equal to the magnitude of its eigenvalues \cite{horn1985matrix}, and since $A$ is Schur stable, we have $\sigma_1(A) = |\lambda_1(A)| < 1$. Further, the similarity transformation $T$ described in Lemma \ref{atrans} can be taken to be the identity matrix. Under these conditions, the bound reduces to \eqref{stablesimplebound}.
\end{proof}

\textbf{Discussion.} Although it is not surprising that such bounds should exist for stable networks, they provide insight into the properties of the dynamics matrix $A$ that affect the potential efficacy of actuator selection. The effect is most clearly seen in Corollary \ref{stableboundcor}, where we observe that the difference between worst and optimal increases as $A$ approaches instability, confirming intuition. The bounds complement those in the previous subsection: here, even the worst $k$-element actuator selection cannot have arbitrarily bad performance for stable networks, whereas even the best selection may incur large costs in unstable networks. However, even in stable networks, effective actuator set selections (perhaps obtained with greedy algorithms \cite{summers2016actuator}) can significantly improve feedback control costs. 


\section{Numerical Experiments}
We now illustrate our results with numerical experiments in regular and random graphs. To build insight and intuition, we focus some of our analysis on an undirected path network, with dynamics matrix 
$$A = \frac{\rho}{3}\left[\begin{array}{ccccc}1 & 1 & 0 & \cdots & 0 \\1 & 1 & 1 & \cdots & \vdots \\0 & 1 & \ddots & \ddots & 0 \\\vdots & \vdots & \ddots & 1 & 1 \\0 & \cdots & 0 & 1 & 1\end{array}\right],  $$
where $\rho > 0 $ is a parameter we will used to modulate the stability of the dynamics. Throughout this section we assume that $\mathcal{B} = \{ e_1,..., e_n\}$, so that each possible actuator injects an input into the dynamics of a single node, and that $Q=I$ and $R_S = I$, $\forall S$. Fig. 1 shows how the optimal feedback performance varies as the number of controlled nodes increases for a $50$-node path network, with varying network stability properties and actuators spaced evenly throughout the path, which is empirically a near optimal actuator placement. We see that when the network becomes unstable, the optimal feedback control costs increase significantly with only a single actuator, even though a single actuator is sufficient to stabilize the network dynamics. 


\begin{figure}
\includegraphics[scale = 0.45]{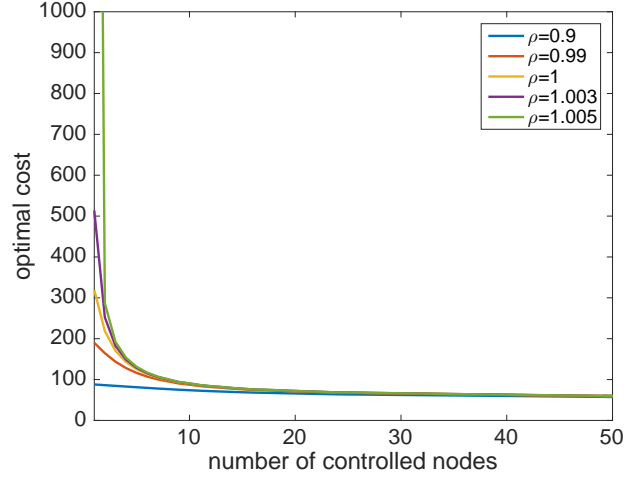}
\caption{Optimal cost versus the number of controlled nodes for a 50-node path graph. The controlled nodes were evenly spaced throughout the path. We see that when the dynamics are stable ($\rho=0.9,0.99,1$) the optimal cost is not too large, even with only a single controlled node. When the dynamics are unstable ($\rho=1.003,1.005$), the optimal cost can be very large.}
\label{fig:path}
\end{figure}



\begin{figure}[t]
\centering
\includegraphics[width=0.9\columnwidth]{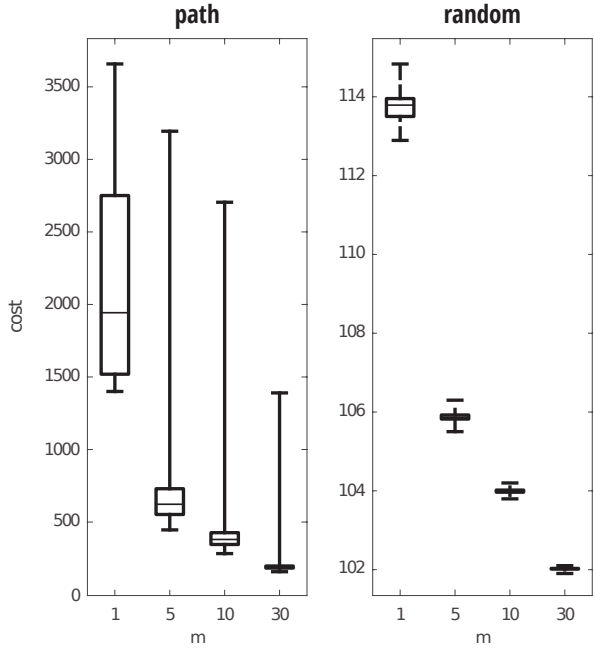}
\caption{For the path graph and Erdos-Renyi random graph ($p=0.1$) with $n=100$ nodes, $m\in\{1,5,10,30\}$ actuators were selected uniformly randomly. Each of the box plots represent a sample distribution of the costs of 1000 samples (realizations of $B$). These box plots demarcate the mean, first and third quartiles (box), and minimum and maximum (whiskers).} \label{fig:randomBs}
\end{figure}

The analytic expressions derived in this paper provide best and worst case cost bounds in different contexts. In Section III.A., for systems that are relatively difficult to regulate (i.e., they have at least one unstable mode), we derive a lower bound on the cost required to regulate the system to the origin for a fixed number of actuators. Similarly, in Section III.B., for systems that are relatively easy to regulate (i.e., all modes are stable), we identify an upper bound on the cost required to regulate the system to the origin for a fixed number of actuators. Effectively, when the system is inherently hard, we quantify the best case cost; when the system is inherently easy, we quantify the worst case cost. These relationships are informative because they reveal the scaling of cost based on the number of actuators. However, it is well-known that $m$ actuators from $n$ nodes can be selected in many ways and that these choices have different costs associated with them. Likewise, the directions associated with unstable modes can dominate the cost, and, therefore, the cost can vary depending on the exact initial state that is required to regulate to the origin, or on the disturbance covariance matrix in the stochastic control case. These questions of actuator selection and target regulation (target control) are not new, however, here we empirically demonstrate the types of variation we observe by using the generalized LQR cost (which has not been studied before).

We first address the variation in the cost for a fixed number of actuators $m$. We observe this variation by selecting $m$ nodes uniformly from $n$, constructing the matrix $B$ (such that the columns of $B$ are columns of the identity matrix), and calculating the LQR cost. We repeat this process 1000 times for each choice of $m\in\{1,5,10,30\}$, constructing the sample distributions in Fig. \ref{fig:randomBs} for the path graph with $n=100$ nodes presented earlier and for the Erdos-Renyi random graph ($p=0.1$). In both cases the adjacency matrix $A$ has been scaled by its largest eigenvalue to make it marginally stable. While the exponential scaling related to the number of actuators can still be observed clearly, there is significant variation in the cost for a specific choice of $m$, most notably for lower fractions of actuators. In addition, the denser connectivity of the random graph yields not only smaller costs, but also smaller variation due to selection of $B$. This implies that the actuator selection problem becomes trivial as the number of actuators or the connectivity increases because all choices will provide roughly equivalent costs.

\begin{figure}[t]
\includegraphics[width=\columnwidth]{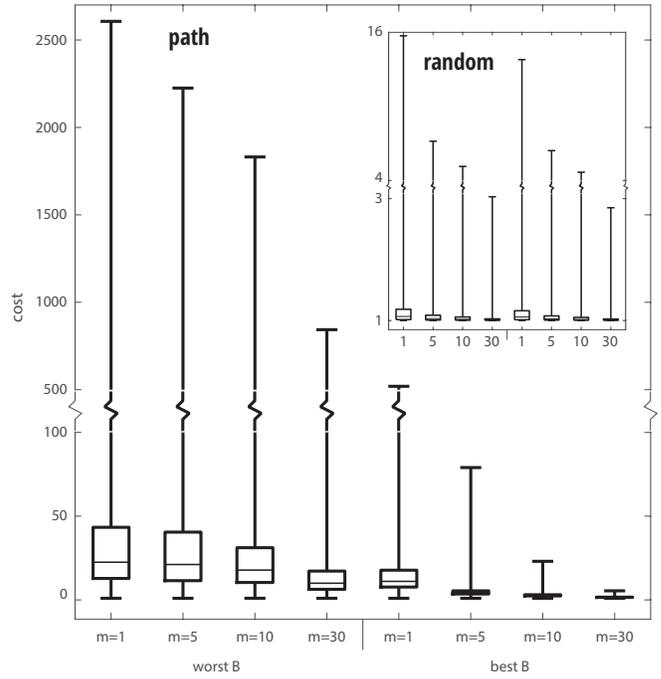}
\caption{For the path graph with $n=100$ nodes, $m\in\{1,5,10,30\}$ actuators were selected in both a greedy (minimizing the cost) and anti-greedy (maximizing the cost) fashion. Each of the box plots represent a sample distribution of the costs associated with 1000 (normally) randomly generated initial states $x_0$ with that $\|x_0\|=1$. The inset plot shows the same results for the Erdos-Renyi random graph. These box plots demarcate the mean, first and third quartiles (box), and minimum and maximum (whiskers).} \label{fig:randomx0}
\end{figure}

\begin{figure*}[t]
\centering
\includegraphics[width=0.9\linewidth]{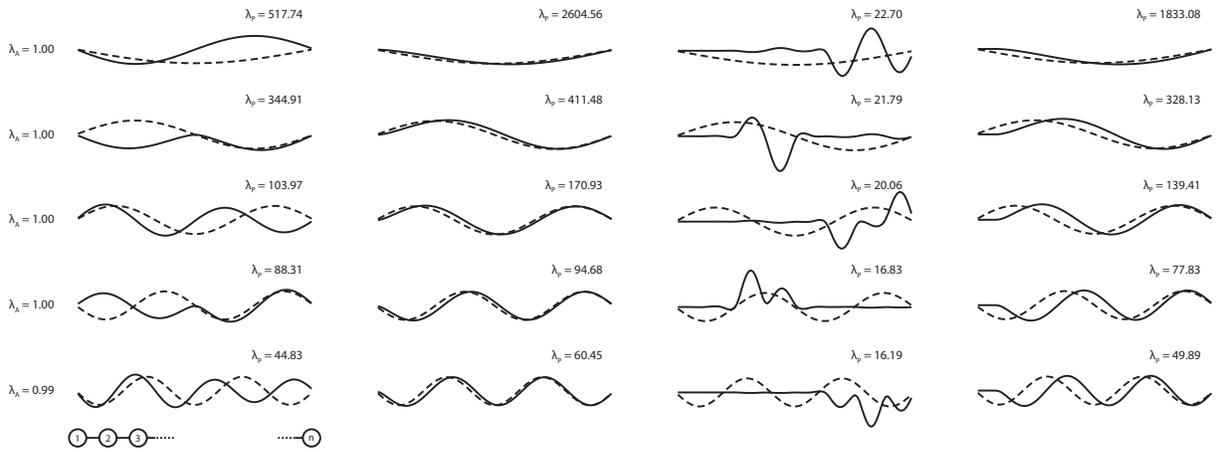}
\caption{For the path graph with $n=100$ nodes, the first five modes (in decreasing absolute value of eigenvalue) of $A$ are plotted in dashed black. Overlaid in solid black are the first five modes of $P_0$ corresponding to choices of $B$ for (from left to right) the single best actuator ($m=1$), the single worst actuator ($m=1$), the 10 best actuators ($m=10$), and the 10 worst actuators ($m=10$). Here ``best'' and ``worst'' are found using a greedy method.}\label{fig:modes}
\end{figure*}

We now turn to look at the variation in the cost caused by selectively choosing certain directions in state space to regulate. For a given number of actuators $m$, we pick the best selection of $m$ actuators and also pick the worst selection of $m$ actuators. We find these (approximate) best and worst case actuators sets by, respectively minimizing and maximizing the cost using a greedy algorithm. For each of these cases, we draw 1000 initial state vectors $x_0$ from a normal distribution, normalize them to lie on the $\|x_0\|=1$ ball, and compute the cost $x_0^TP_0x_0$ for regulating that specific direction. Fig. \ref{fig:randomx0} displays these sample distributions for the $n=100$ path graph for $m\in\{1,5,10,30\}$ actuators. The inset plot shows the same for the Erdos-Renyi random graph. By selecting the best and worst actuator choices, we have captured the extreme cases due to actuator selection; every other choice of $B$ would fall (roughly) in between, falling in line with the results of Fig. \ref{fig:randomBs}. We observe that ideal actuator selection results in a system that has significantly less variation due to direction. More specifically, the optimal choice of actuators eliminates, or greatly reduces, the effect of the most unstable modes present in $A$.

One way to interpret the distributions in Fig. \ref{fig:randomx0} is that we know the directions that are most and least costly to regulate - these are the eigenvectors (modes) of $P_0$ corresponding, respectively, to the largest and smallest absolute eigenvalues of $P_0$. For a given box and whisker, the maximum value is attained at $v_1^TP_0v_1$, where $v_1$ is the eigenvector corresponding to $\lambda_1$ of $P_0$ and we have ordered our eigenvalues such that $\lambda_1 \geq \lambda_2 \geq \dots \geq \lambda_n$. Likewise, the direction cheapest to regulate is $v_n$, which is the minimum of the distribution captured by the box plots. All other directions fall between these extremes. 

To see this more clearly, in Fig. \ref{fig:modes}, we plot the first five modes of $A$ and $P_0$ for the path graph (again ordering the eigenvectors according to descending absolute value of their corresponding eigenvalue) for best and worst actuator selection with $m=1$ and $m=10$. The eigenvectors of $A$ (dashed) encode the modes expressed in the dynamics due to the network structure and the eigenvectors of $P_0$ encode the directions in state space can break down the overall LQR cost. The best placed single actuator lies at the middle of the path, whereas the worst lies at one of the ends.  We observe that the ideal actuator changes the modes of the path network substantially whereas the worst actuator choice does not change the modes, indicating that an actuator placed at the end of the path does not have a significant impact on the dynamics of the network. The largest eigenvalues in the best and worst case differ by approximately a factor of four. The effect is exaggerated in the $m=10$ case, where the best actuators are evenly spaced throughout the path and the worst actuators are all aggregated at one end. A similar pattern is observed with respect to mode shape and the difference in the largest eigenvalue of $P_0$ is about a factor of 80.

\section{Conclusions}
We have derived a set of performance bounds for optimal feedback control in networks that provide insight into fundamental difficulties of network control as a function of the dynamics structure and control architecture. Ongoing and future work includes deriving tighter and more general bounds to include input effectiveness and logarithmic capacity of dynamics eigenvalues \cite{olshevsky2016eigenvalue}, studying similar properties for dynamic game performance metrics, and conducting more elaborate case studies. 

\pagestyle{empty}

\bibliographystyle{IEEEtran}
\bibliography{refs.bib}

\begin{thebibliography}{10}
\providecommand{\url}[1]{#1}
\csname url@samestyle\endcsname
\providecommand{\newblock}{\relax}
\providecommand{\bibinfo}[2]{#2}
\providecommand{\BIBentrySTDinterwordspacing}{\spaceskip=0pt\relax}
\providecommand{\BIBentryALTinterwordstretchfactor}{4}
\providecommand{\BIBentryALTinterwordspacing}{\spaceskip=\fontdimen2\font plus
\BIBentryALTinterwordstretchfactor\fontdimen3\font minus
  \fontdimen4\font\relax}
\providecommand{\BIBforeignlanguage}[2]{{%
\expandafter\ifx\csname l@#1\endcsname\relax
\typeout{** WARNING: IEEEtran.bst: No hyphenation pattern has been}%
\typeout{** loaded for the language `#1'. Using the pattern for}%
\typeout{** the default language instead.}%
\else
\language=\csname l@#1\endcsname
\fi
#2}}
\providecommand{\BIBdecl}{\relax}
\BIBdecl

\bibitem{liu2011controllability}
Y.-Y. Liu, J.-J. Slotine, and A.-L. Barab{\'a}si, ``Controllability of complex
  networks,'' \emph{Nature}, vol. 473, no. 7346, pp. 167--173, 2011.

\bibitem{rajapakse2011dynamics}
I.~Rajapakse, M.~Groudine, and M.~Mesbahi, ``Dynamics and control of
  state-dependent networks for probing genomic organization,''
  \emph{Proceedings of the National Academy of Sciences}, vol. 108, no.~42, pp.
  17\,257--17\,262, 2011.

\bibitem{cowan2012nodal}
N.~Cowan, E.~Chastain, D.~Vilhena, J.~Freudenberg, and C.~Bergstrom, ``Nodal
  dynamics, not degree distributions, determine the structural controllability
  of complex networks,'' \emph{PLOS ONE}, vol.~7, no.~6, p. e38398, 2012.

\bibitem{nepusz2012controlling}
T.~Nepusz and T.~Vicsek, ``Controlling edge dynamics in complex networks,''
  \emph{Nature Physics}, vol.~8, no.~7, pp. 568--573, 2012.

\bibitem{wang2012optimizing}
W.-X. Wang, X.~Ni, Y.-C. Lai, and C.~Grebogi, ``Optimizing controllability of
  complex networks by minimum structural perturbations,'' \emph{Physical Review
  E}, vol.~85, no.~2, p. 026115, 2012.

\bibitem{ruths2014control}
J.~Ruths and D.~Ruths, ``Control profiles of complex networks,''
  \emph{Science}, vol. 343, no. 6177, pp. 1373--1376, 2014.

\bibitem{olshevsky2014minimal}
A.~Olshevsky, ``Minimal controllability problems,'' \emph{IEEE Transactions on
  Control of Network Systems}, vol.~1, no.~3, pp. 249--258, 2014.

\bibitem{pequito2016}
S.~Pequito, S.~Kar, and A.~Aguiar, ``A framework for structural input/output
  and control configuration selection in large-scale systems,'' \emph{IEEE
  Transactions on Automatic Control}, vol.~61, no.~2, pp. 303--318, 2016.

\bibitem{pasqualetti2014controllability}
F.~Pasqualetti, S.~Zampieri, and F.~Bullo, ``Controllability metrics,
  limitations and algorithms for complex networks,'' \emph{Control of Network
  Systems, IEEE Transactions on}, vol.~1, no.~1, pp. 40--52, 2014.

\bibitem{summers2014submodularity}
T.~Summers, F.~Cortesi, and J.~Lygeros, ``On submodularity and controllability
  in complex dynamical networks,'' \emph{IEEE Transactions on Control of
  Network Systems}, vol.~3, no.~1, pp. 1--11, 2016.

\bibitem{summers2014optimal}
T.~Summers and J.~Lygeros, ``Optimal sensor and actuator placement in complex
  dynamical networks,'' in \emph{IFAC World Congress, Cape Town, South Africa},
  2014, pp. 3784--3789.

\bibitem{yannature2015}
G.~Yan, G.~Tsekenis, B.~Barzel, J.-J. Slotine, Y.-Y. Liu, and A.-L. Barab\'asi,
  ``Spectrum of controlling and observing complex networks,'' \emph{Nature
  Physics}, vol.~11, pp. 779--786, 2015.

\bibitem{tzoumas2016}
V.~Tzoumas, M.~A. Rahimian, G.~Pappas, and A.~Jadbabaie, ``Minimal actuator
  placement with bounds on control effort,'' \emph{to appear, IEEE Transactions
  on Control of Network Systems}, 2016.

\bibitem{summers2016actuator}
T.~Summers, ``Actuator placement in networks using optimal control performance
  metrics,'' in \emph{IEEE Conference on Decision and Control}.\hskip 1em plus
  0.5em minus 0.4em\relax IEEE, 2016, pp. 2703--2708.

\bibitem{polyak2013lmi}
B.~Polyak, M.~Khlebnikov, and P.~Shcherbakov, ``An lmi approach to structured
  sparse feedback design in linear control systems,'' in \emph{European Control
  Conference}.\hskip 1em plus 0.5em minus 0.4em\relax IEEE, 2013, pp. 833--838.

\bibitem{munz2014sensor}
U.~Munz, M.~Pfister, and P.~Wolfrum, ``Sensor and actuator placement for linear
  systems based on and optimization,'' \emph{IEEE Transactions on Automatic
  Control}, vol.~59, no.~11, pp. 2984--2989, 2014.

\bibitem{dhingra2014admm}
N.~K. Dhingra, M.~R. Jovanovic, and Z.-Q. Luo, ``An {ADMM} algorithm for
  optimal sensor and actuator selection,'' in \emph{IEEE Conference on Decision
  and Control}.\hskip 1em plus 0.5em minus 0.4em\relax IEEE, 2014, pp.
  4039--4044.

\bibitem{zhang2017sensor}
H.~Zhang, R.~Ayoub, and S.~Sundaram, ``Sensor selection for kalman filtering of
  linear dynamical systems: Complexity, limitations and greedy algorithms,''
  \emph{Automatica}, vol.~78, pp. 202--210, April, 2017.

\bibitem{lancaster1995algebraic}
P.~Lancaster and L.~Rodman, \emph{Algebraic {R}iccati equations}.\hskip 1em
  plus 0.5em minus 0.4em\relax Clarendon press, 1995.

\bibitem{komaroff1994iterative}
N.~Komaroff, ``Iterative matrix bounds and computational solutions to the
  discrete algebraic {R}iccati equation,'' \emph{IEEE Transactions on Automatic
  Control}, vol.~39, no.~8, pp. 1676--1678, 1994.

\bibitem{horn1985matrix}
R.~A. Horn and C.~R. Johnson, \emph{Matrix analysis}.\hskip 1em plus 0.5em
  minus 0.4em\relax Cambridge university press, 1985.

\bibitem{liu2011open}
J.~Liu and J.~Zhang, ``The open question of the relation between square
  matrix's eigenvalues and its similarity matrix's singular values in linear
  discrete system,'' \emph{International Journal of Control, Automation and
  Systems}, vol.~9, no.~6, pp. 1235--1241, 2011.

\bibitem{olshevsky2016eigenvalue}
A.~Olshevsky, ``Eigenvalue clustering, control energy, and logarithmic
  capacity,'' \emph{Systems \& Control Letters}, vol.~96, pp. 45--50, 2016.

\end{thebibliography}

\end{document}